\documentclass[12pt]{article}
\usepackage[intlimits]{amsmath}
\usepackage{amsfonts,amsthm}
\usepackage{array}
\usepackage{amscd}
\usepackage{amssymb}
\usepackage{enumerate}

\newtheorem{theorem}{Theorem}
\newtheorem{lemma}{Lemma}

\newtheorem{remark}{Remark}

\newcounter{ex}

\input{tcilatex}
\begin{document}

\title{A note on a diffeomorphism between two Banach spaces }
\author{Marek Galewski, Elzbieta Galewska \ and Ewa Schmeidel}
\maketitle

\begin{abstract}
\noindent We provide sufficient conditions for a mapping acting between two
Banach spaces to be a diffeomorphism.

\noindent \textbf{Math Subject Classifications}: 57R50, 58E05

\end{abstract}

\section{Introduction}

Given two Banach spaces $X$ and $B$, a continuously Fr\'{e}chet -
differentiable map $f:X\rightarrow B$ is called a diffeomorphism if it is a
bijection and its inverse $f^{-1}:B\rightarrow X$ is continuously Fr\'{e}%
chet - differentiable as well. A continuous linear mapping $\Lambda
:X\rightarrow B$, $\Lambda \in L\left( X,B\right) $, is a Fr\'{e}chet -
derivative of $f$ at $x\in X$ provided that for all $h\in X$ it holds that%
\begin{equation}
f\left( x+h\right) -f\left( x\right) =\Lambda h+o\left( \left\Vert
h\right\Vert \right)  \label{def-FRE}
\end{equation}%
and where $\lim_{\left\Vert h\right\Vert \rightarrow 0}\frac{\left\Vert
o\left( \left\Vert h\right\Vert \right) \right\Vert }{\left\Vert
h\right\Vert }=0$; $\Lambda $ is then typically denoted as $f^{\prime }(x)$
while its action on $h$ as $f^{\prime }(x)h$. Mapping $f$ is continuously Fr%
\'{e}chet - differentiable if $f^{\prime }:X\rightarrow L\left( X,B\right) $
is continuous in respective topologies. Obviously if a mapping $f$ is a
diffeomorphism, it is automatically a homeomorphism, while the vice versa is
not correct as seen by example of a function $f\left( x\right) =x^{3}$.
Recalling the Inverse Function Theorem a continuously Fr\'{e}chet -
differentiable mapping $f:X\rightarrow B$ such that for any $x\in X$ the
derivative is surjective, i.e. $f^{\prime }(x)X=H$ and invertible, i.e.
there exists a constant $\alpha _{x}>0$ such that%
\begin{equation*}
\left\Vert f^{\prime }(x)h\right\Vert \geq \alpha _{x}\left\Vert h\right\Vert
\end{equation*}%
defines a local diffeomorphism. This means that for each point $x$ in $X$,
there exists an open set $U$ containing $x$, such that $f(U)$ is open in $B$
and $\left. f\right\vert _{U}:U\rightarrow f(U)$ is a diffeomorphism. If $f$
is a diffeomorphism it obviously defines a local diffeomorphism. Thus the
main problem to be overcome is to make a local diffeomorphism a global one.
Or in other words: what assumptions should be imposed on the spaces involved
and the mapping $f$ to have global diffeomorphism from the local one. This
task can be investigated within the critical point theory, or more precisely
with mountain geometry.

Such research has apparently been started by Katriel \cite{katriel}. His
result can be summarized as follows, see also Theorem 5.4 from \cite{jabri}:

\begin{theorem}
\label{MainTheo copy(2)}Let $X,$ $B$ be finite dimensional Euclidean spaces.
Assume that $f:X\rightarrow B$ is a $C^{1}$-mapping such that\newline
(a1) $f^{\prime }(x)$ is invertible for any $x\in X$;\newline
(a2) $\left\Vert f\left( x\right) \right\Vert \rightarrow \infty $ as $%
\left\Vert x\right\Vert \rightarrow \infty $\newline
then $f$ is a diffeomorphism.
\end{theorem}

Recently, Idczak, Skowron and Walczak \cite{SIW} using the Mountain Pass
Lemma and ideas contained in the proof of Theorem \ref{MainTheo copy(2)}
(see \cite{jabri} for some nice version) proved the result concerning
diffeomorphism between a Banach and a Hilbert space. They further applied
this abstract tool to the initial value problem for some
integro-differential system in order to get differentiability of the
solution operator. It seems that differentiable dependence on parameters for
boundary value problems can be investigated by this method. The result from 
\cite{SIW} reads\textbf{:}

\begin{theorem}
\label{TheoSW}Let $X$ be a real Banach space, $H$ - a real Hilbert space. If 
$f:X\rightarrow H$ is a $C^{1}$-mapping such that\newline
(b1) for any $y\in H$ the functional $\varphi :X\rightarrow 
\mathbb{R}
$ given by the formula 
\begin{equation*}
\varphi \left( x\right) =\frac{1}{2}\left\Vert f\left( x\right)
-y\right\Vert ^{2}
\end{equation*}%
satisfies Palais-Smale condition;\newline
(b2) for any $x\in X$, $f^{\prime }(x)X=H$ and there exists a constant $%
\alpha _{x}>0$ such that%
\begin{equation}
\left\Vert f^{\prime }(x)h\right\Vert \geq \alpha _{x}\left\Vert h\right\Vert
\label{alfa_x}
\end{equation}%
\newline
then $f$ is a diffeomorphism.
\end{theorem}

The question aroused \textit{whether the Hilbert space }$H$\textit{\ in the
formulation of the above theorem could be replaced by a Banach space.} This
question is of some importance since one would expect diffeomorphism to act
between two Hilbert spaces or two Banach spaces rather than between a
Hilbert and a Banach space. The applications given in \cite{SIW} work when
both $X$ and $H$ are Hilbert spaces.

The aim of this note is to provide an affirmative answer to this question.
We also simplify a bit the proof of Theorem \ref{TheoSW} by using a weak
version of the MPL Lemma due to Figueredo and Solimini, see \cite{fig1}, 
\cite{fig2} which we recall below.

Functional $J:X\rightarrow \mathbb{R}$ satisfies the Palais-Smale condition
if every sequence $(u_{n})$ such that $\{J(u_{n})\}$ is bounded and $%
J^{\prime }(u_{n})\rightarrow 0$, has a convergent subsequence. We note that
in a finite dimensional setting condition (\textit{a2}) implies that the
Palais-Smale condition holds for $x\rightarrow \left\Vert f\left( x\right)
\right\Vert $. The version of the Mountain Pass Lemma (MPL Lemma) which we
use is as follows.

\begin{lemma}
\cite{fig1}\label{MPT}(\textbf{Mountain Pass Lemma}) Let $X$ be a Banach
space and $J\in C^{1}(X,\mathbb{R})$ satisfies the Palais-Smale condition.
Assume that 
\begin{equation}
\inf_{\left\Vert x\right\Vert =r}J(x)\geq \max \{J(0),J(e)\}\text{,}
\label{condMPT}
\end{equation}%
where $0<r<\left\Vert e\right\Vert $ and $e\in X$. Then $J$ has a non-zero
critical point $x_{0}$.
\end{lemma}

\begin{remark}
\label{remafertMPT}From the proof of Lemma \ref{MPT}\ it is seen that if $%
\inf_{\left\Vert x\right\Vert =r}J(x) > \max \{J(0),J(e)\}$, then also $%
x\neq e $.
\end{remark}

\section{Main result}

Our main result concerns extension of Theorem \ref{TheoSW}\ to the case of $%
H $ being a Banach space. We retain the assumption providing local
diffeomorphism and modify assumption (\textit{b1}) to get the global
diffeomorphism. This is realized by replacing $\left\Vert \cdot \right\Vert
^{2}$ with some functional $\eta $ for which functional $x\rightarrow \eta
\left( f\left( x\right) -y\right) $ satisfies the Palais-Smale condition for
all $y.$ One can think of $\eta $ as $\eta \left( x\right)
=\int_{0}^{1}\left\vert x\left( t\right) \right\vert ^{p}dt$ for $x\in
L^{p}\left( 0,1\right) $, $p>1$. Our main result reads\textbf{:}

\begin{theorem}
\label{MainTheo}Let $X,$ $B$ be real Banach spaces. Assume that $%
f:X\rightarrow B$ is a $C^{1}$-mapping, $\eta :B\rightarrow 
\mathbb{R}
_{+}$ is a $C^{1}$ functional and that the following conditions hold\newline
(c1) $\left( \eta \left( x\right) =0\Longleftrightarrow x=0\right) $ and $%
\left( \eta ^{^{\prime }}\left( x\right) =0\Longleftrightarrow x=0\right) $; 
\newline
(c2) for any $y\in B$ the functional $\varphi :X\rightarrow 
\mathbb{R}
$ given by the formula 
\begin{equation*}
\varphi \left( x\right) =\eta \left( f\left( x\right) -y\right)
\end{equation*}%
satisfies Palais-Smale condition;\newline
(c3) for any $x\in X$ the Fr\'{e}chet derivative is surjective, i.e. $%
f^{\prime }(x)X=B$, and there exists a constant $\alpha _{x}>0$ such that
for all $h\in X$%
\begin{equation*}
\left\Vert f^{\prime }(x)h\right\Vert \geq \alpha _{x}\left\Vert
h\right\Vert \text{;}
\end{equation*}%
\newline
(c4) there exist positive constants $\alpha $, $c$, $M$ such that 
\begin{equation*}
\eta \left( x\right) \geq c\left\Vert x\right\Vert ^{\alpha }\text{ for }%
\left\Vert x\right\Vert \leq M
\end{equation*}%
\newline
then $f$ is a diffeomorphism.
\end{theorem}

\begin{proof}
We follow the ideas used in the proof of Main Theorem in \cite{SIW} with
necessary modifications. In view of the remarks made in the Introduction
condition (\textit{c3}) implies that $f$ is a local diffeomorphism. Thus it
is sufficient to show that $f$ is onto and one to one.\bigskip 

Firstly we show that $f$ is onto.\textbf{\ }Let us fix any point $y\in B$.
Observe that $\varphi $ is a composition of two $C^{1}$ mappings, thus $%
\varphi \in C^{1}\left( X,%
\mathbb{R}
\right) $. Moreover, $\varphi $ is bounded from below and satisfies the
Palais-Smale condition. Thus from the Ekeland's Variational Principle it
follows that there exists argument of a minimum which we denote by $%
\overline{x}$, see Theorem 4.7 \cite{fig1}. We see by the chain rule for Fr%
\'{e}chet derivatives and by Fermat's Principle that 
\begin{equation*}
\varphi ^{^{\prime }}(\overline{x})=\eta ^{^{\prime }}(f\left( \overline{x}%
\right) -y)\circ f^{\prime }(\overline{x})=0\text{.}
\end{equation*}%
Since by (\textit{c3)} mapping $f^{\prime }(\overline{x})$ is invertible we
see that $\eta ^{^{\prime }}(f\left( \overline{x}\right) -y)=0$. Now by (c1)
it follows that%
\begin{equation*}
f\left( \overline{x}\right) -y=0\text{.}
\end{equation*}%
Thus $f$ is surjective.\bigskip 

Now we argue by contradiction that $f$ is one to one.\textbf{\ }Suppose
there are $x_{1}$ and $x_{2}$, $x_{1}\neq x_{2}$, $x_{1}$, $x_{2}\in X$,
such that $f\left( x_{1}\right) =f\left( x_{2}\right) =a\in B$. We will
apply Lemma \ref{MPT}. Thus we put $e=x_{1}-x_{2}$ and define mapping $%
g:X\rightarrow B$ by the following formula%
\begin{equation*}
g\left( x\right) =f\left( x+x_{2}\right) -a\text{.}
\end{equation*}%
Observe that $g\left( 0\right) =g\left( e\right) =0$. We define functional $%
\psi :X\rightarrow 
\mathbb{R}
$ by the following formula%
\begin{equation*}
\psi \left( x\right) =\eta \left( g\left( x\right) \right) .
\end{equation*}%
By (\textit{c2}) functional $\psi $ satisfies the Palais-Smale condition.
Next we see that $\psi \left( e\right) =\psi \left( 0\right) =0.$ Using (\ref%
{def-FRE}) and (\ref{alfa_x}) we see that there is a number $\rho >0$ such
that%
\begin{equation}
\frac{1}{2}\alpha _{\overline{x}}\left\Vert x\right\Vert \leq \left\Vert
g\left( x\right) \right\Vert \text{ for }x\in \overline{B\left( 0,\rho
\right) }\text{.}  \label{ro}
\end{equation}%
Indeed, since $\lim_{\left\Vert h\right\Vert \rightarrow 0}\frac{o\left(
\left\Vert h\right\Vert \right) }{\left\Vert h\right\Vert }=0$ we see that
for $\left\Vert h\right\Vert $ sufficiently small$,$say $\left\Vert
h\right\Vert \leq \delta $, it holds that $o\left( \left\Vert h\right\Vert
\right) \leq \frac{1}{2}\alpha _{x_{2}}\left\Vert h\right\Vert $ and 
\begin{equation*}
g\left( 0+h\right) -g\left( 0\right) =g^{^{\prime }}(0)h+o\left( \left\Vert
h\right\Vert \right) .
\end{equation*}%
By definition of $g$ and by (\textit{c3}) we see for $\left\Vert
h\right\Vert \leq \delta $ that%
\begin{equation*}
\left\Vert g\left( h\right) \right\Vert +\frac{1}{2}\alpha
_{x_{2}}\left\Vert h\right\Vert \geq \left\Vert g\left( h\right) -o\left(
\left\Vert h\right\Vert \right) \right\Vert =\left\Vert f^{^{\prime
}}(x_{2})h\right\Vert \geq \alpha _{x_{2}}\left\Vert h\right\Vert .
\end{equation*}%
We can always assume that $\delta <\rho <\min \left\{ \left\Vert
e\right\Vert ,M\right\} $. Thus (\ref{ro}) holds. Take any $0<r<\rho $.
Recall that by (\textit{c4}) we obtain since (\ref{ro}) holds 
\begin{equation*}
\psi (x)=\eta \left( g\left( x\right) \right) \geq c\left\Vert g\left(
x\right) \right\Vert ^{\alpha }\geq c\left( \frac{1}{2}\alpha
_{x_{2}}\right) ^{\alpha }\left\Vert x\right\Vert ^{\alpha }.
\end{equation*}%
Thus 
\begin{equation*}
\inf_{\left\Vert x\right\Vert =r}\psi (x)\geq c\left( \frac{1}{2}\alpha
_{x_{2}}\right) ^{\alpha }\left\Vert r\right\Vert ^{\alpha }>0=\psi \left(
e\right) =\psi \left( 0\right) 
\end{equation*}%
We see that \textbf{(}\ref{condMPT}\textbf{)} is satisfied for $J=\psi $.
Thus by Lemma \ref{MPT} and by Remark \ref{remafertMPT} we note that $\psi $
has a critical point $v\neq 0$, $v\neq e$ and such that 
\begin{equation*}
\psi ^{^{\prime }}(v)=\eta ^{^{\prime }}(f\left( v+x_{2}\right) -a)\circ
f^{\prime }(v+x_{2})=0.
\end{equation*}%
Since $f^{\prime }(v+x_{2})$ is invertible, we see that $\eta ^{^{\prime
}}(f\left( v+x_{2}\right) -a)=0$. So by the assumption (c1) we calculate $%
f\left( v+x_{2}\right) -a=0$. This means that either $v=0$ or $v=e$. Thus we
obtain a contradiction which shows that $f$ is a one to one operator.
\end{proof}

We supply our result with a few of remarks.

\begin{remark}
We see that from Theorem \ref{MainTheo} by putting $\eta \left( x\right) =%
\frac{1}{2}\left\Vert x\right\Vert ^{2}$ we obtain easily Theorem \ref%
{TheoSW}. In that case $c=1$, $M>0$ is arbitrary, $\alpha =2$. It seems
there is no difference as concerns the finite and infinite dimensional
context.
\end{remark}

\begin{remark}
Since the deformation lemma is also true with Cerami condition, we can
assume that $\varphi $ satisfies the Cerami condition instead of the
Palais-Smale condition. However, in the possible applications, in which the
A-R condition could not be assumed, it seems that checking the Palais-Smale
condition would be an easier task. We refer to \cite{KRV}, \cite{MR} for
some other variational methods.
\end{remark}

\section{Conclusion and other results}

We would like to mention \cite{zampieri} for some other approach connected
with the nonnegative auxiliary scalar coercive function and the main
assumption that for all positive $r$ : $\sup_{\left\Vert x\right\Vert \leq
r}\left\Vert f^{^{\prime }}(x)^{-1}\right\Vert <+\infty $ and $\left\Vert
f(x)\right\Vert \rightarrow +\infty $ as $\left\Vert x\right\Vert
\rightarrow +\infty $. The methods of the proof are quite different as well.
One of the results of \cite{zampieri} most closely connected to ours and to
those of \cite{SIW} reads as follows

\begin{theorem}
Let $X,$ $B$ be a real Banach spaces. Assume that $f:X\rightarrow B$ is a $%
C^{1}$-mapping, $\left\Vert f(x)\right\Vert \rightarrow +\infty $ as $%
\left\Vert x\right\Vert \rightarrow +\infty $, for all $x\in X$ $f^{^{\prime
}}(x)\in Isom\left( X,B\right) $ and for all $x\in X$ $\sup_{\left\Vert
x\right\Vert \leq r}\left\Vert f^{^{\prime }}(x)^{-1}\right\Vert <+\infty $
for all $r>0$. Then $f$ is a diffeomorphism.
\end{theorem}

The main difference between our results and the existing one is that we do
not require condition $\sup_{\left\Vert x\right\Vert \leq r}\left\Vert
f^{^{\prime }}(x)^{-1}\right\Vert <+\infty $ for all $r>0$. We have
boundedness of $\left\Vert f^{^{\prime }}(x)^{-1}\right\Vert $ but in a
pointwise manner. Still it seems that checking the condition $\left\Vert
f(x)\right\Vert \rightarrow +\infty $ as $\left\Vert x\right\Vert
\rightarrow +\infty $ might be difficult in a direct manner. Recall that $%
\varphi \left( x\right) =\eta \left( f\left( x\right) -y\right) $ is bounded
from below, $C^{1}$ and satisfies the Palais-Smale condition and therefore
it is coercive as well. However, coercivity alone does not provide the
existence of exactly one minimizer. We would have to add strict convexity to
the assumptions. Thus we can obtain easily the following result

\begin{theorem}
Let $X,$ $B$ be a real Banach spaces. Assume that $f:X\rightarrow B$ is a $%
C^{1}$-mapping, $\eta :B\rightarrow 
\mathbb{R}
_{+}$ is a $C^{1}$ functional and that the following conditions hold\newline
(d1) $\left( \eta \left( x\right) =0\Longleftrightarrow x=0\right) $ and $%
\left( \eta ^{^{\prime }}\left( x\right) =0\Longleftrightarrow x=0\right) $. 
\newline
(d2) for any $y\in B$ the functional $\varphi :X\rightarrow 
\mathbb{R}
$ given by the formula 
\begin{equation*}
\varphi \left( x\right) =\eta \left( f\left( x\right) -y\right)
\end{equation*}%
is coercive and strictly convex;\newline
(d3) for any $x\in X$ the Fr\'{e}chet derivative is surjective, i.e. $%
f^{^{\prime }}(x)X=B$, and there exists a constant $\alpha _{x}>0$ such that
for all $h\in X$%
\begin{equation*}
\left\Vert f^{^{\prime }}(x)h\right\Vert \geq \alpha _{x}\left\Vert
h\right\Vert
\end{equation*}%
\newline
\newline
then $f$ is a diffeomorphism.
\end{theorem}

\begin{proof}
Let us fix $y\in B$. Note that by (\textit{d2)} $\varphi $ has exactly one
minimizer $\overline{x}$. Thus by Fermat's Principle we see that 
\begin{equation*}
\varphi ^{^{\prime }}(\overline{x})=\eta ^{^{\prime }}(f\left( \overline{x}%
\right) -y)\circ f^{^{\prime }}(\overline{x})=0\text{.}
\end{equation*}%
Since by (\textit{d3}) mapping $f^{^{\prime }}(\overline{x})$ is invertible
we see that $\eta ^{^{\prime }}(f\left( \overline{x}\right) -y)=0$. Now by (%
\textit{d1}) it follows that%
\begin{equation*}
f\left( \overline{x}\right) -y=0\text{.}
\end{equation*}%
Thus $f$ is surjective and obviously one to one since $\overline{x}$ is
unique.
\end{proof}

We believe that checking that $\varphi $ is strictly convex is still more
demanding than proving that $\varphi $ satisfies the Palais-Smale condition.

\begin{tabular}{l}
Marek Galewski \\ 
Institute of Mathematics, \\ 
Technical University of Lodz, \\ 
Wolczanska 215, 90-924 Lodz, Poland, \\ 
marek.galewski@p.lodz.pl%
\end{tabular}

\begin{tabular}{l}
El\.{z}bieta Galewska \\ 
Centre of Mathematics and Physics, \\ 
Technical University of Lodz, \\ 
Al. Politechniki 11, 90-924 Lodz, Poland, \\ 
elzbieta.galewska@p.lodz.pl%
\end{tabular}

\begin{tabular}{l}
Ewa Schmeidel \\ 
Faculty of Mathematics and Computer Science, \\ 
University of Bialystok, \\ 
Akademicka 2, 15-267 Bia\l ystok, Poland, \\ 
eschmeidel@math.uwb.edu.pl%
\end{tabular}

\end{document}